\documentclass[11pt,a4paper]{article}

\usepackage{mathtools}
\usepackage{authblk} 
\usepackage{algpseudocode,algorithmicx,algorithm}

\usepackage{mathrsfs}
\usepackage{latexsym,bm}
\usepackage{amsmath,amsfonts,amssymb,amsthm}
\usepackage{extarrows}

\usepackage{graphicx,subfigure,epstopdf,float}
\usepackage{enumerate,cases,multirow}
\usepackage{makecell}
\usepackage{caption}

\usepackage{longtable,colortbl,arydshln,threeparttable}
\definecolor{mygray}{gray}{.9}

\usepackage{indentfirst}
\usepackage[top=25mm,bottom=20mm,left=25mm,right=20mm]{geometry}
\usepackage{cite}

\usepackage{listings}

\usepackage{makeidx}        
\usepackage{booktabs}
\usepackage[bookmarks,bookmarksnumbered,colorlinks,citecolor=red,linkcolor=red,hyperindex,linktocpage=true]{hyperref}

\newcommand{\ket}[1]{| #1 \rangle} 
\newcommand{\bra}[1]{\langle #1 |} 

\newcommand{\bb}{\boldsymbol}

\def \d {\mathrm{d}}
\def \e {\mathrm{e}}
\def \i {\mathrm{i}}

\newcounter{parentalgorithm}

\makeatother

\newtheorem{theorem}{Theorem}[section]
\newtheorem{lemma}{Lemma}[section]

\newtheorem{definition}{Definition}[section]

\theoremstyle{remark}
\newtheorem{remark}{\bf Remark}[section]

\numberwithin{equation}{section}

\usepackage{qcircuit}
\begin{document}

\title{\bf Schr\"odingerization based quantum algorithms for regularized Wasserstein proximal operators}
\author[1]{Shi Jin\thanks{shijin-m@sjtu.edu.cn}}
\author[1, 2]{Nana Liu\thanks{nana.liu@quantumlah.org}}
\affil[1]{School of Mathematical Sciences, Institute of Natural Sciences, MOE-LSC, Shanghai Jiao Tong University, Shanghai, 200240, China}
\affil[2]{Global College, Shanghai Jiao Tong University, Shanghai 200240, China}
\author[3,4,5]{Yue Yu\thanks{terenceyuyue@xtu.edu.cn}}
\affil[3]{School of Mathematics and Computational Science, Xiangtan University, Xiangtan, Hunan 411105, China}
\affil[4]{Hunan Research Center of the Basic Discipline Fundamental Algorithmic Theory and Novel Computational Methods, Xiangtan, Hunan 411105, China}
\affil[5]{National Center for Applied Mathematics in Hunan, Xiangtan, Hunan 411105, China}
\maketitle

\begin{abstract}
We develop a quantum algorithm for the regularized Wasserstein proximal operator, which is a fundamental tool in optimal transport and mean-field games. The regularization introduces a small diffusive term into the continuity equation of the Benamou-Brenier formulation, which results in a forward-backward PDE system consisting of a Fokker-Planck equation and a viscous Hamilton-Jacobi equation with a quadratic Hamiltonian. Through the Cole-Hopf transformation, both equations are converted to forward heat equations, whose coupling requires a Hadamard division to prepare the initial data for the second heat equation and a Hadamard product to recover the terminal density.
We solve these heat equations via the Schr\"odingerization method and implement the Hadamard division and product operations using simple matrix-vector multiplication representations.
 The complete quantum algorithm prepares an $\varepsilon$-approximation of the terminal density state with $\mathcal{O}(d N_x T \log^2(1/\varepsilon))$ query complexity, up to constants depending on the potential and initial density, where $d$ is the spatial dimension, $N_x$ is the number of grid points per spatial dimension and $T$ is the evolution time. The complexity depends only {\it linearly} on $d N_x$, yielding an {\it exponential} speedup over classical methods, whose cost scales as $N_x^d$ per time step.  Numerical experiments validate the effectiveness of the proposed algorithm.
\end{abstract}

\section{Introduction}

The Wasserstein proximal operator (WPO) naturally models how a distribution evolves under the influence of an energy landscape while strictly respecting the geometry of optimal transport. It serves as a powerful engine for sampling, optimization, special applications in mean-field games, and generative modeling \cite{Villani2009,Gu2024W1W2,Han2025TensorTrain,Han2025Splitting}.

In statistics and machine learning, a major goal is to draw samples from a complex target probability distribution $p(x) \propto \e^{-V(x)}$, where $V(x)$ is an unnormalized potential energy. The celebrated JKO scheme \cite{jordan1998variational} formulates the classical diffusion and Fokker-Planck equation as a gradient flow of the Boltzmann entropy in the Wasserstein space, where each step is realized by solving a WPO  problem. Similar problems also arise in potential (or variational) mean-field games \cite{LasryLions2007,HuangMalhameCaines2006}, where the agent running cost derives from a global population functional $\mathcal{F}(\rho)$ (interaction energy, congestion, entropy etc.). Standard mean-field game equilibrium can be obtained by solving coupled Hamilton-Jacobi-Bellman system, in which the WPO is also the time discretization or the JKO scheme of the corresponding  Wasserstein gradient flow.

In generative AI, modern models such as Stable Diffusion and Sora transport a simple noise distribution through a learned velocity field to generate realistic data. Recent work \cite{Zhang2024WPO} establishes a fundamental connection between score-based generative models and WPOs, showing that the forward-backward structure of the Wasserstein proximal system is mathematically identical to the reverse-time stochastic differential equations used in diffusion models \cite{Zhang2024WPO,lin2021wasserstein}.

 In stochastic optimal control, the classical Schr\"odinger bridge problem (SBP) \cite{Caluya2022SBP} seeks to steer the joint state probability density function between two prescribed endpoint distributions while minimizing the expected control effort subject to controlled diffusion. Its variational formulation shares the same dynamic optimal transport structure as the Benamou-Brenier formulation, but with the continuity equation replaced by a Fokker-Planck equation, where the diffusion coefficient serves as the diffusive regularization. The regularized WPO \cite{LLO2023} further extends this framework by incorporating an additional terminal cost functional, thereby generalizing the SBP to a broader class of density control problems.

In \cite{benamou2000computational},  Benamou and Brenier provided a fluid dynamical formulation for the Monge-Kantorovich mass transfer problem, whose solution can be obtained by solving a coupled continuity equation and a Hamilton-Jacobi equation with a terminal condition for the density, which needs to be solved iteratively.  For the WPO, however, only their decoupled system needs to be solved.  The regularized WPO, introduced in \cite{LLO2023}, adds a diffusion regularization to the continuity equation, yielding a forward-backward PDE system consisting of a Fokker-Planck equation and a viscous Hamilton-Jacobi equation with a quadratic Hamiltonian. This enables the application of the classical Cole-Hopf transformation that equivalently transfers the system into two linear heat equations, one forward and one backward in time. However, when the spatial dimension is high, classical computation becomes daunting, which motivates our development of quantum algorithms.


Our quantum algorithms are based on the viscously regularized WPOs introduced in \cite{LLO2023}, which are combined with the Schr\"odingerization method \cite{JLY22SchrLong,JLY22SchrShort,analogPDE,JLMY2025SchOptimal}. The latter is a simple and general methodology to unitarize general linear ordinary and partial differential equations into unitary evolution processes, thus making them suitable for quantum simulation.

 The overall quantum algorithm proceeds as follows.
\begin{itemize}
  \item First, through the Cole-Hopf transformation, the forward-backward PDE system is reduced to two forward heat equations of the form $\partial_t u = \beta \Delta u$, where $u$ represents either $\eta$ in \eqref{heatHJ} or $\psi$ in \eqref{heatFP}. We solve these heat equations using the Schr\"odingerization method, whose query complexity for preparing an $\varepsilon$-approximation of the normalized state $|\bb{u}(T)\rangle$ is
\[\mathcal{O}\Big(\frac{\|\bb{u}(0)\|}{\|\bb{u}(T)\|} \, d N_x \beta T \log \frac{\|\bb{u}(0)\|}{\varepsilon \|\bb{u}(T)\|}\Big),\]
where $N_x$ is the number of grid points per spatial dimension, $d$ is the spatial dimension, $T$ is the stopping time, $\beta$ is the artificial viscosity coefficient, and $\varepsilon$ is the precision.
  \item Second, to couple the two heat equations, we implement a quantum Hadamard division to prepare the initial data for the second heat equation from the output of the first.
  \item Third, we recover the terminal density via a quantum Hadamard product.
\end{itemize}

The overall query complexity contains, in addition to the $d N_x \beta T$ factor from the heat equation solver, several normalization factors and a condition number $\kappa_{\bb{\eta}(T)} = \max_j |\eta(T,x_j)| / \min_j |\eta(T,x_j)|$ from the Hadamard division and product steps. We carefully analyze these quantities and show that, under the assumption $0 \le V(x) \le V^*$, they are bounded by constants depending only on the range of the potential:
\[
\frac{\|\bb{\eta}_0\|}{\|\bb{\eta}(T)\|} \lesssim \e^{V^*/(2\beta)},\quad
\kappa_{\bb{\eta}(T)} \lesssim \e^{V^*/(2\beta)},\quad
g_{\mathrm{prod}} \cdot \frac{\|\bb{\psi}_0\|}{\|\bb{\psi}(T)\|}
\lesssim \e^{3V^*/(2\beta)} \|\rho_0\|_{L^2},
\]
where $g_{\mathrm{prod}}$ is the Hadamard product normalization factor. These estimates naively introduce exponential factors in $\beta^{-1}$. However, by exploiting the minimizer-invariance scaling $V \mapsto V/G,\; T \mapsto TG$ with $G = 1/\beta$, which leaves the WPO unchanged, we remove this $\beta$-dependence entirely. Under this scaling, the complete quantum algorithm prepares an $\varepsilon$-approximation of the normalized terminal density state $|\bb{\rho}(T)\rangle$ with $\Omega(1)$ success probability, using
\[
\mathcal{O} \Big(\e^{3V^*/2} \|\rho_0\|_{L^2} \, d N_x T \log^2 \frac{1}{\varepsilon} \Big)
\]
queries to the block-encoding of the discrete Laplacian, and
\[
\mathcal{O}(\e^{3V^*/2} \|\rho_0\|_{L^2} \log \frac{1}{\varepsilon}\Big)
\]
queries to the state preparation oracles $U_{\rho_0}$ and $U_{\eta_0}$. Crucially, the complexity depends only {\it linearly } on $d$ and $N_x$.

For comparison, a classical Fourier spectral method for the same discretization requires $\mathcal{O}(N_x^d \log N_x)$ operations per time step for each heat equation, with additional $\mathcal{O}(N_x^d)$ operations for the Hadamard division and product. This classical cost exhibits an exponential dependence on the spatial dimension $d$ through $N_x^d$. In contrast, the quantum complexity is linear in $d$ and $N_x$, offering an {\it exponential} advantage for high-dimensional regularized WPOs.

The remainder of this paper is organized as follows. In Section~\ref{sec:problem}, we formulate the problem and provide an overview of our algorithm, highlighting the reduction to forward heat equations via Cole-Hopf transformations. Section~\ref{sec:quantum} presents the quantum algorithm for the regularized WPO  via Schr\"odingerization, which consists of three parts: quantum Hadamard product and division protocols, the Schr\"odingerization-based heat equation solver, and the assembly of these components into the complete algorithm along with its complexity analysis. Numerical experiment is provided in Section~\ref{sec:num}, and concluding remarks and outlooks are given in Section~\ref{sec:conclusion}.

\section{Problem description and algorithm overview} \label{sec:problem}

Proximal operators are fundamental tools in optimization. Their counterpart in the Wasserstein-2 space, known as the WPO, has proven to be highly effective in various applications, including scientific computing and machine learning. In this section, we review the approach proposed in \cite{LLO2023} for approximating this operator.

\subsection{The regularized WPO}

 To make this paper accessible to readers from quantum computation and other fields, we begin by introducing the WPO in intuitive terms, before presenting its mathematical formulation and the regularization that enables our quantum algorithm.

\paragraph{Euclidean proximal operators.}
In classical optimization, a proximal operator is a simple but powerful tool, which takes a point and move it a small step in the direction that decreases the energy while staying close to the original point. Specifically, given a current point $x_k$ and an energy function $F$, the proximal step finds a new point that balances two competing objectives: staying close to $x_k$, and decreasing $F$. This is achieved by solving
\[
x_{k+1} = \arg\min_x \Big\{ \frac{1}{2\tau} \|x - x_k\|^2 + F(x) \Big\},
\]
where the first term penalizes moving too far from the current point, and the second term encourages lowering the energy. The parameter $\tau$ controls the step size. This idea is the foundation of many optimization algorithms, including gradient descent and the proximal gradient method.

When working with probability distributions, it is often useful to describe how a distribution evolves over time under the influence of some energy or objective functional. The WPO extends the idea of proximal operators from points to probability distributions, and from Euclidean distance to the Wasserstein distance from optimal transport.

\paragraph{Wasserstein distance.}

We now replace the Euclidean distance $\|\cdot\|$ by the Wasserstein distance $\mathcal{W}(\cdot, \cdot)$ from optimal transport, which measures the minimal cost of moving mass from one distribution to another. Formally, for two probability distributions $\rho$ and $q$ on $\mathbb{R}^d$, the squared {Wasserstein-2 distance} is defined as
\[
\mathcal{W}(\rho, q)^2 = \inf_{\pi \in \Pi(\rho, q)} \int_{\mathbb{R}^d \times \mathbb{R}^d} \|x - y\|^2 \,\mathrm{d}\pi(x,y),
\]
where $\Pi(\rho, q)$ denotes the set of all joint distributions on $\mathbb{R}^d \times \mathbb{R}^d$ whose marginals are $\rho$ and $q$. Intuitively, $\pi(x,y)$ specifies how much mass from location $x$ in the initial distribution is transported to location $y$ in the target distribution, and the integral sums the squared distance of each such movement.

The Wasserstein distance has found widespread applications across mathematics, statistics, and engineering \cite{Villani2009}. It is used to study optimal approximation of measures by finite point sets, to compare color distributions in image processing, and to analyze mixing and convergence of Markov chains. In statistical mechanics, it plays a central role in the theory of propagation of chaos and the mean-field behavior of large particle systems. One of its key advantages is that it behaves well in high-dimensional settings, making it particularly suitable for problems in infinite-dimensional spaces, such as the long-time behavior of stochastic partial differential equations and the hydrodynamic limits of interacting particle systems. In Riemannian geometry, the Wasserstein-2 distance is closely tied to the study of Ricci curvature through its connection to diffusion equations.

\paragraph{Benamou-Brenier formulation of the Wasserstein distance.}

The Wasserstein-2 distance admits a dynamic reformulation as an optimal control problem, known as the Benamou-Brenier formulation \cite{benamou2000computational}:
\[
\frac{\mathcal{W}(\rho_0, q)^2}{2T}
=
\inf_{\substack{\rho, v \\ \rho(0)=\rho_0,\ \rho(T)=q}}
\int_0^T \int_{\mathbb{R}^d} \frac{1}{2} \|v(t,x)\|^2 \rho(t,x) \d x \d t,
\]
where the infimum is taken over density functions $\rho: [0,T] \times \mathbb{R}^d \to \mathbb{R}_+$ and vector fields $v: [0,T] \times \mathbb{R}^d \to \mathbb{R}^d$ satisfying the continuity equation
\[
\partial_t \rho(t,x) + \nabla \cdot \big( \rho(t,x) v(t,x) \big) = 0,
\]
with
\[
\rho(0, x) = \rho_0(x), \qquad \rho(T, x) = q(x).
\]

 A crucial distinction for the WPO from the general optimal transport problem is that the terminal density $q$ here is not prescribed in advance; rather, it is itself  a distribution  to be optimized. This seemingly subtle difference has profound consequences: as we shall see, it decouples the Hamilton-Jacobi equation from the continuity equation in the optimality system, ultimately allowing both equations to be linearized via the Cole-Hopf transformation when viscosously regularized.

\paragraph{Wasserstein proximal operators.}

Suppose we want to update not a point, but a probability distribution $\rho_k$~---~a cloud of mass representing, for example, the location of a population, the state of a physical system, or the distribution of data in machine learning. We would like to decrease some energy functional $\mathcal{F}(\rho)$~---~such as entropy, potential energy, or congestion cost~---~while ensuring that the new distribution does not stray too far from the original one. This leads to the {Wasserstein proximal operator}:
\[
\rho_{k+1} = \arg\min_{\rho} \Big\{ \frac{1}{2\tau} \mathcal{W}^2(\rho, \rho_k) + \mathcal{F}(\rho) \Big\}.
\]
This operator performs a gradient-descent step in the space of probability distributions, where distance is measured by optimal transport. In other words, the next distribution is the best compromise between staying close to the current distribution (measured by optimal transport cost) and reducing the energy.

This construction, introduced in the seminal JKO scheme \cite{jordan1998variational}, provides a natural time discretization of Wasserstein gradient flows and plays a central role in optimal transport, diffusion equations, and mean field game models.

\paragraph{The regularized Wasserstein proximal operator.}

In this work, we focus on a specific case where the energy functional is linear:
\[\mathcal{F}(q) = \int_{\mathbb{R}^d} V(x) q(x) \d x,\]
where $V\in C^1(\mathbb{R}^d)$ is a given potential function. For this choice, the Wasserstein proximal operator is defined as
\begin{equation}\label{Wprox}
\rho_T:= \text{WProx}_{T,\mathcal{F}}(\rho_0) := \arg\min_{q \in \mathcal{P}_2(\mathbb{R}^d)} \mathcal{F}(q) + \frac{\mathcal{W}(\rho_0,q)^2}{2T},
\end{equation}
where $\mathcal{W}(\rho_0,q)$ is the Wasserstein-2 distance between $\rho_0$ and $q$, and $\mathcal{P}_2(\mathbb{R}^d)$ is the probability density set with finite second moment. The parameter $T$ controls the trade-off between transport cost and energy reduction and can be interpreted as the time horizon over which the transport occurs.

Following \cite{LLO2023}, we consider a {regularized} version of this operator, where a diffusion term is added to the dynamics. The regularization serves two important purposes: it makes the problem well-posed and numerically stable, and~---~crucially for our purposes~---~it allows the optimality conditions to be decoupled and linearized. The regularized WPO is defined by the mean-field control problem
\[
\rho_T = \text{WProx}_{T,\mathcal{F},\beta}(\rho_0) := \arg\min_{\rho,\nu,q} \int_0^T \int_{\mathbb{R}^d} \frac{1}{2}\|v(t,x)\|^2 \rho(t,x) \d x\d t + \int_{\mathbb{R}^d} V(x) q(x) \d x,
\]
subject to the Fokker-Planck equation
\[
\partial_t \rho(t,x) + \nabla\cdot(\rho(t,x) v(t,x)) = \beta \Delta \rho(t,x),\qquad \rho(0,x)=\rho_0(x),\qquad \rho(T,x)=q(x),
\]
 where $\beta>0$ is the small regularization parameter--the  viscosity coefficient. This regularization is critical, since it allows the use of the Cole-Hopf transformation, as will be seen later.

\paragraph{Optimality conditions and the forward-backward system.}


To derive the optimality conditions,  we introduce a Lagrange multiplier $\Phi(t,x)$ to enforce the Fokker-Planck constraint. The Lagrangian is
\[
\mathcal{L}(\rho, v, \Phi)
=
\int_0^T \int_{\mathbb{R}^d}
\Big( \frac12 \|v\|^2 \rho
+ \Phi \Big( \partial_t \rho + \nabla\cdot(\rho v) - \beta \Delta \rho \Big) \Big)
\,\mathrm{d}x\,\mathrm{d}t
+ \int_{\mathbb{R}^d} V(x) \rho(T,x) \,\mathrm{d}x.
\]
Taking variations with respect to $\rho$ and $v$ yields the optimality conditions (see Proposition 2 of \cite{LLO2023}). In particular, variation with respect to $v$ gives $v = \nabla \Phi$: the optimal velocity is the gradient of the Lagrange multiplier. Substituting this into the associated Euler-Lagrange equations gives the forward-backward PDE system
\begin{align}
\partial_t \rho + \nabla\cdot(\rho \nabla \Phi) = \beta \Delta \rho,  \label{LLO1} \\
\partial_t \Phi + \frac12 |\nabla \Phi|^2 = -\beta \Delta \Phi, \label{LLO2}
\end{align}
with
\[\rho(0,x)=\rho_0(x), \qquad \Phi(T,x)=-V(x).\]


 According to Proposition 2 of \cite{LLO2023},
\[\rho_T := \text{WProx}_{T,\mathcal{F},\beta}(\rho_0) = \rho(T,x).\]

 The structure of this system is special: the Hamilton-Jacobi equation \eqref{LLO2} is a backward equation with a known terminal condition $\Phi(T, x) = -V(x)$. Because it does not depend on $\rho$, it can be solved independently (by reversing time). Once $\Phi$ is known, the Fokker-Planck equation \eqref{LLO1}  becomes a linear equation for $\rho$ that can be solved forward in time from the initial condition $\rho_0$. This decoupling is the key that makes the problem amenable for {\it exact linearization} by the Cole-Hopf transformation, which we present in the next section.

\subsection{The heat equation form via Cole-Hopf transformation}

The second equation \eqref{LLO2} is a backward Hamilton-Jacobi equation with a terminal condition. By setting $S_{\beta}(t,x) = -\Phi(T-t,x)$, one obtains an initial value problem for $S_\beta$:
    \begin{equation}\label{LLO2S}
    \partial_t S_{\beta} + \frac12 |\nabla S_{\beta}|^2 = \beta \Delta S_{\beta}, \qquad S_{\beta}(0,x) = V(x),
    \end{equation}
which is precisely the viscous Hamilton-Jacobi equation with a quadratic Hamiltonian. It is simple to check that the Cole-Hopf transformation
    \[
    \eta(t,x) = \e^{-S_\beta(t,x)/(2\beta)}
    \]
converts \eqref{LLO2S} into the heat equation
    \begin{equation}\label{heatHJ}
    \partial_t \eta = \beta \Delta \eta, \qquad \eta(0,x) = \e^{-V(x)/(2\beta)}=:\eta_0(x).
    \end{equation}

The first equation \eqref{LLO1} is a Fokker-Planck equation. By using the Cole-Hopf type transformation $\psi(t,x) = \rho(t,x) \e^{-\Phi(t,x)/(2\beta)}$, one readily verifies that $\psi$ also satisfies the heat equation
    \begin{equation}\label{heatFP}
    \partial_t \psi = \beta \Delta \psi,
    \end{equation}
    with initial condition
    \[
    \psi(0,x) = \rho_0(x) \e^{-\Phi(0,x)/(2\beta)} = \rho_0(x) \e^{S_\beta(T,x)/(2\beta)} = \frac{\rho_0(x)}{\eta(T,x)}.
    \]

Therefore, one can first solve \eqref{heatHJ} to obtain $\eta(T,x)$, then solve \eqref{heatFP} to obtain $\psi(T,x)$, and finally recover the terminal density via
\[
\rho(T,x) = \psi(T,x) \e^{-V(x)/(2\beta)} = \psi(T,x) \eta_0(x).
\]

\begin{remark}\label{rem:hadamard}
Let $\{x_j\}_{j=1}^{N}$ be a set of discrete points. Define the vectors
\[
\bb{\rho}_0 := \bigl(\rho_0(x_j)\bigr)_{j=1}^{N}, \quad
\bb{\eta}(T) := \bigl(\eta(T,x_j)\bigr)_{j=1}^{N}, \quad
\bb{\psi}(T) := \bigl(\psi(T,x_j)\bigr)_{j=1}^{N}, \quad
\bb{\eta}_0 := \bigl(\eta_0(x_j)\bigr)_{j=1}^{N}.
\]
For the quantum implementation, the difficulty lies in the computation of the componentwise (Hadamard) division $\bb{\rho}_0 \oslash \bb{\eta}(T)$ and the componentwise (Hadamard) product $\bb{\psi}(T) \odot \bb{\eta}_0$, i.e., the vectors whose $j$-th components are $\rho_0(x_j)/\eta(T,x_j)$ and $\psi(T,x_j)\,\eta_0(x_j)$, respectively.
\end{remark}

\begin{remark} \label{rem:bd}
We assume that the density function $\rho$ has compact support or nearly vanishes near the boundary of a sufficiently large finite spatial domain, so that the periodic boundary condition can be imposed for computational convenience without introducing significant boundary effects. The algorithm can be readily extended to other boundary conditions as well.
\end{remark}

\section{Quantum algorithm for regularized Wasserstein proximal operators via Schr\"odingerization} \label{sec:quantum}

The heat equations \eqref{heatHJ} and \eqref{heatFP} can be solved via the Schr\"odingerization approach. The overall procedure requires quantum Hadamard division and product operations, which are used to prepare the initial data for coupling the two heat equations and to reconstruct the terminal density state.

\subsection{Quantum Hadamard product and division}

Block encoding is a general input model for matrix operations on a quantum computer~\cite{Gilyen2019QSVD,Chakraborty2019blockEncode,Lin2022Notes}, with the definition given below.

\begin{definition}\label{def:blockencoding}
	Suppose that $ A $ is an $ n $-qubit matrix and let $ \Pi = \bra{0^m} \otimes I $ with $ I $ being an $ n $-qubit identity matrix. If there exist positive numbers $ \alpha $ and $ \varepsilon $, and a unitary matrix $ U_A $ of $ (m+n) $-qubits, such that
	\[
	\| A - \alpha \Pi U_A \Pi^\dag \| = \| A - \alpha (\bra{0^m} \otimes I) U_A (\ket{0^m} \otimes I) \| \le \varepsilon,
	\]
	then $ U_A $ is called an $ (\alpha, m, \varepsilon) $ block-encoding of $ A $.
\end{definition}

Our construction relies on the diagonal block encoding of amplitudes described below.
\begin{lemma}[Theorem 2 of \cite{Rattew2023nonlinearTrans}] \label{lem:diag-block-encoding}
Let $U_\psi$ be an $n$-qubit unitary that prepares the quantum state $|\psi\rangle_n = U_\psi|0\rangle_n = \sum_{j=0}^{N-1} \psi_j \ket{j}_n$ with complex amplitudes $\psi_j$. Then one can construct a $(1, n+3, 0)$-block-encoding $U_{D_\psi}$ for the diagonal operator $D_\psi = \text{diag}(\psi_0, \psi_1, \cdots, \psi_{N-1})$. The construction requires circuit depth $\mathcal{O}(n)$ in terms of single and two qubit gates and makes $\mathcal{O}(1)$ calls to a controlled-$U_{\psi}$ gate.
\end{lemma}

From the explicit construction of the diagonal block encoding in Theorem~2 of \cite{Rattew2023nonlinearTrans}, we observe that the entire circuit can be implemented using only single- and two-qubit gates. In particular, the controlled-copy circuit $C$ is realized with $n$ Toffoli gates, each decomposable into a constant number of elementary gates. All other components (e.g., $W_p$, $G_p$, $\hat{Z}$, $\hat{H}$, $\hat{S}$) contribute at most $\mathcal{O}(n)$ elementary gates. Consequently, the total number of elementary gates is $\mathcal{O}(n)$, and the circuit depth is also $\mathcal{O}(n)$ under parallel execution.

Let
\[
\bb{a} = [a_0, a_1, \cdots, a_{N-1}]^\top, \qquad
\bb{b} = [b_0, b_1, \cdots, b_{N-1}]^\top \in \mathbb{C}^N,
\]
and denote their corresponding quantum states as
\[\ket{\bb{a}} = \frac{1}{\|\bb{a}\|}\sum_{j=0}^{N-1} a_j \ket{j}, \qquad
\ket{\bb{b}} = \frac{1}{\|\bb{b}\|}\sum_{j=0}^{N-1} b_j \ket{j}.\]
We assume that the state preparation unitaries $U_{\bb{a}}$ and $U_{\bb{b}}$ satisfying
$U_{\bb{a}}|0^n\rangle = \ket{\bb{a}}$ and $U_{\bb{b}}|0^n\rangle = \ket{\bb{b}}$ are given.

Our construction for the Hadamard product and division is based on the following observation:
\begin{itemize}
  \item \textbf{Hadamard product.} Observe that the componentwise product can be written as
\[
\bb{a} \odot \bb{b} = D_{\bb{a}} \bb{b}, \qquad
D_{\bb{a}} := \text{diag}(a_0, a_1, \cdots, a_{N-1}).
\]
Thus, up to normalization, the quantum state encoding $\bb{a} \odot \bb{b}$ is given by $D_{\bb{a}} \ket{\bb{b}}$.
  \item \textbf{Hadamard division.} Similarly, the componentwise quotient (assuming $b_j \neq 0$ for all $j$) is
\begin{equation}\label{Haddiv}
\bb{a} \oslash \bb{b} = D_{\bb{b}}^{-1} \bb{a},
\end{equation}
i.e., $D_{\bb{b}}^{-1} \ket{\bb{a}}$ up to normalization.
Hence the task reduces to inverting the diagonal operator $D_{\bb{b}}$.
\end{itemize}

\begin{theorem}[Quantum Hadamard product and division]
\label{thm:hadamard}
Let $U_{\bb{a}}$ and $U_{\bb{b}}$ be $n$-qubit state preparation unitaries for the vectors $\bb{a}, \bb{b} \in \mathbb{C}^N$ with $N = 2^n$, as defined above. Then:

\begin{enumerate}
\item[(i)] \textbf{Hadamard product.}
There exists a quantum circuit that produces the normalized state encoding $\bb{a} \odot \bb{b}$. The circuit uses $\mathcal{O}(g_{\text{prod}})$ calls to controlled-$U_{\bb{a}}$ and $U_{\bb{b}}$ and has circuit depth $\mathcal{O}(n g_{\text{prod}})$ in terms of single- and two-qubit gates, where
\[
g_{\text{prod}} = \frac{\|\bb{a}\| \|\bb{b}\|}{\|\bb{a} \odot \bb{b}\|}.
\]

\item[(ii)] \textbf{Hadamard division.}
For any $\epsilon > 0$, there exists a quantum circuit that produces a state $\epsilon$-close (in $\ell_2$-norm) to the normalized state encoding $\bb{a} \oslash \bb{b}$. The circuit requires $\mathcal{O}(g_{\text{div}})$ calls to controlled-$U_{\bb{b}}$ and $U_{\bb{a}}$, and its circuit depth is $\mathcal{O}(n g_{\text{div}})$, where
\[g_{\text{div}} = \kappa_{\bb{b}} \log \frac{1}{\epsilon}, \qquad  \kappa_{\bb{b}} = \frac{\max_j |b_j|}{\min_j |b_j|}. \]
\end{enumerate}
\end{theorem}

\begin{proof}
\textbf{Part (i).}
By Lemma~\ref{lem:diag-block-encoding}, from the normalized state $\ket{\bb{a}} = U_{\bb{a}}\ket{0^n}$ we can construct a $(1, n+3, 0)$-block-encoding $U_{D_{\bb{a}}}$ of the diagonal matrix $D_{\ket{\bb{a}}} = \text{diag}(a_0/\|\bb{a}\|, \cdots, a_{N-1}/\|\bb{a}\|)$. By definition, $U_{D_{\bb{a}}}$ is also a $(\|\bb{a}\|, n+3, 0)$-block-encoding of $D_{\bb{a}} = \text{diag}(a_0, \cdots, a_{N-1})$.

Prepare the initial state $\ket{0}^{\otimes (n+3)} \otimes \ket{\bb{b}}$ by $U_{\bb{b}}$. Applying $U_{D_{\bb{a}}}$ yields
\[
U_{D_{\bb{a}}} \big( \ket{0}^{\otimes (n+3)} \otimes \ket{\bb{b}} \big) = \frac{1}{\|\bb{a}\|} \ket{0}^{\otimes (n+3)} \otimes D_{\bb{a}} \ket{\bb{b}} + \ket{\Phi^\perp},
\]
where $\ket{\Phi^\perp}$ is orthogonal to the subspace where the first $n+3$ qubits are in $\ket{0}^{\otimes (n+3)}$. Substituting $\ket{\bb{b}} = \bb{b} / \|\bb{b}\|$, we obtain
\[
U_{D_{\bb{a}}} \big( \ket{0}^{\otimes (n+3)} \otimes \ket{\bb{b}} \big) = \frac{1}{\|\bb{a}\| \|\bb{b}\|} \ket{0}^{\otimes (n+3)} \otimes D_{\bb{a}} \bb{b} + \ket{\Phi^\perp}.
\]
Measuring the ancillary register and post-selecting on the all-zero outcome collapses the system to $\ket{D_{\bb{a}} \bb{b}} =\ket{\bb{a} \odot \bb{b}}$. The success probability is
\[\text{P}_{\text{prod}} = \Big(\frac{\| D_{\bb{a}} \bb{b} \|}{ \|\bb{a}\| \|\bb{b}\|}\Big)^2 = \Big(\frac{\|\bb{a} \odot \bb{b}\| }{\|\bb{a}\| \|\bb{b}\|}\Big)^2.\]
 The circuit depth and query complexity follow directly from Lemma~\ref{lem:diag-block-encoding} and the amplitude amplification.

\textbf{Part (ii).}
Following the same procedure as in part (i), we can construct a $(\|\bb{b}\|, n+3, 0)$-block-encoding $U_{D_{\bb{b}}}$ of $D_{\bb{b}} = \text{diag}(b_0, \cdots, b_{N-1})$. According to Theorem 12 of \cite{Costa2021QLSA}, there exists a quantum linear systems algorithm (QLSA) that produces the normalized state $\ket{D_{\bb{b}}^{-1} \bb{a}} = \ket{\bb{a} \oslash \bb{b}}$ within an error $\epsilon$, using $\mathcal{O}(\kappa_{\bb{b}} \log \frac{1}{\epsilon})$ calls to the oracles $U_{D_{\bb{b}}}$ and $U_{\bb{a}}$, where $\kappa_{\bb{b}} = \frac{\max_j |b_j|}{\min_j |b_j|}$ is the condition number of $D_{\bb{b}}$.
By Lemma \ref{lem:diag-block-encoding}, the query complexity to controlled-$U_{\bb{b}}$ is $\mathcal{O}(\kappa_{\bb{b}} \log(1/\epsilon))$, and the circuit depth is multiplied by the same factor because each QSVT step calls the block-encoding once. This completes the proof.
\end{proof}

\begin{remark}
While quantum circuits for the Hadamard product using CNOT gates and post-selection exist \cite{Ramezani2023product},  with complexity comparable to ours (cf. Proposition~1 of \cite{HuangN2026product}), the Hadamard division counterpart appears to be absent in the literature.
\end{remark}

\subsection{Schr\"odingerization based quantum algorithm for the heat equation} \label{subsec:Schr}

According to the previous discussion, the problem reduces to solving the heat equation
\begin{equation}\label{heatbeta}
\partial_t u(t,x) = \beta \Delta u(t,x), \qquad u(0,x) = u_0(x)
\end{equation}
on the spatial domain $\Omega = [a,b]^d$, where $a, b, \beta > 0$ are constants. The periodic boundary conditions are imposed in what follows.
Our goal is to solve the regularized Wasserstein proximal problem, where $\beta$ corresponds to the regularization parameter and $u$ is either $\eta$ or $\psi$ (see \eqref{heatHJ} and \eqref{heatFP}).

The spatial discretization is performed using the Fourier spectral method, which leads to a system of linear ODEs that can be solved via the Schr\"odingerization approach.

We consider a $d$-dimensional domain $[a,b]^d$ with a uniform spatial step $\Delta x = (b-a)/N_x$, where $N_x = 2M_x$ is an even integer. The grid points are defined as
\[
x_{\bb{j}} = (x_{j_1}, \cdots, x_{j_d}), \qquad x_{j_i} = a + j_i \Delta x, \quad j_i = 0,1,\cdots,N_x-1, \quad i = 1,\cdots,d.
\]
For a spatial coordinate $x = (x_1,\cdots,x_d) \in [a,b]^d$, we adopt the Fourier basis functions
\[
\phi_{\bb{l}}(x) = \prod_{i=1}^d \phi_{l_i}(x_i), \qquad
\phi_{l_i}(x_i) = \e^{\i \nu_{l_i} (x_i - a)}, \quad
\nu_{l_i} = \frac{2\pi (l_i - M_x)}{b-a}, \quad l_i = 0,1,\cdots,N_x-1,
\]
where $\bb{l} = (l_1,\cdots,l_d)$ and $0 \le \bb{l} \le N_x-1$ (meaning each component satisfies $0 \le l_i \le N_x-1$). The spectral approximation of the solution is expressed as
\begin{equation}\label{Fexpand}
u(t,x) = \sum_{\bb{l}} c_{\bb{l}}(t) \phi_{\bb{l}}(x), \qquad x = x_{\bb{j}}, \quad 0 \le \bb{j} \le N_x-1,
\end{equation}
where the coefficients $c_{\bb{l}}(t)$ are determined by the collocation values at the grid points. Periodic boundary conditions are automatically satisfied by this expansion. The grid values are organized into a column vector as $\bb{u}(t) = \sum_{\bb{j}} u(t,x_{\bb{j}}) |\bb{j}\rangle$, where $|\bb{j}\rangle = |j_1\rangle \otimes \cdots \otimes |j_d\rangle$. The vector of spectral coefficients is similarly given by $\bb{c}(t) = \sum_{\bb{l}} c_{\bb{l}}(t) |\bb{l}\rangle$. The transformation between the physical and spectral representations is $\bb{u} = \Phi^{\otimes d} \bb{c}$, where $\Phi = (\phi_{jl})_{N_x\times N_x}$, with $\phi_{jl} = \phi_l(x_j)$, is the matrix representation of the discrete Fourier transform.

The semi-discretization of \eqref{heatbeta} reads:
\[\frac{\d \bb{u}(t)}{\d t} = \beta(\bb{P}_1^2 + \cdots + \bb{P}_d^2) \bb{u}(t), \qquad \bb{u}(0) = \bb{u}_0, \]
where $\bb{P}_l = I^{\otimes (l-1)} \otimes P_\nu \otimes I^{\otimes (d-l)}$, and $P_\nu = \Phi D_\nu \Phi^{-1}$, with $D_\nu = \text{diag}(\nu_0, \nu_1, \cdots, \nu_{N_x-1})$, is the matrix representation of the momentum operator $\hat{p} = -\i \frac{\partial}{\partial x}$. Let $\bb{c}(t) = F_x^{-1} \bb{u}(t)$, where $F_x = \Phi^{\otimes d}$. The ODEs can be rewritten as
\begin{equation}\label{ODEc}
\frac{\d \bb{c}(t)}{\d t} = A \bb{c}(t), \qquad \bb{c}(0) = F_x^{-1}\bb{u}_0,
\end{equation}
where
\[ A = \beta\Big((\bb{D}_1^\nu)^2 + \cdots + (\bb{D}_d^\nu)^2\Big), \qquad \bb{D}^\nu _l = I^{\otimes (l-1)} \otimes D_\nu \otimes I^{\otimes (d-l)}.\]

For convenience, we set $N_x=2^{n_x}$. The matrix $A$ is diagonal and can be constructed from the block-encoding of $D_\nu$ by using the LCU procedure. According to \cite{Gilyen2019QSVD, Chakraborty2019blockEncode, Lin2022Notes}, we can construct an $(\alpha_A, m_A, 0)$-block-encoding of $A$, with $\alpha_A = d \beta \nu_{\max}$, $m_A = \mathcal{O}(d n_x)$ and $\nu_{\max} = \max_l |\nu_l|$.

The ODE system \eqref{ODEc} can be solved via the Schr\"odingerization approach \cite{JLY22SchrLong,JLY22SchrShort,JLMY2025SchOptimal}, which we briefly review below. Applying the warped phase transformation $\bb{w}(t,p) = \e^{-p}\bb{c}(t)$ for $p \ge 0$, with the equation extended naturally to $p < 0$, yields
\begin{equation}\label{uf2v}
    \frac{\partial}{\partial t} \bb{w}(t) = -A \partial_p \bb{w}.
\end{equation}
The initial data for \eqref{uf2v} is given by $\bb{w}(0,p) = \psi(p) \bb{c}(0)$, where $\psi(p) \in H^r(\mathbb{R})$ with $r \ge 1$ is chosen such that $\psi(p) = \e^{-p}$ for $p \ge 0$ and decays sufficiently fast as $p \to -\infty$. A simple choice is $\psi(p) = \e^{-|p|} \in H^1(\mathbb{R})$. A construction of $\psi$ with higher regularity is provided in \cite{JLM24SchrBackward}. More recently, in \cite{JLMY2025SchOptimal}, we presented a construction that achieves near-optimal query complexity using cut-off function and error function techniques. For numerical implementation, we truncate the auxiliary variable $p$ to a finite interval $[-R,R]$ with $R>0$ satisfying the conditions given therein. We then choose a uniform mesh size $\Delta p = 2R/N_p$, where $N_p = 2^{n_p}$ is an even number, and denote the grid points by $-R = p_0 < p_1 < \cdots < p_{N_p} = R$.
The circuit for implementing the quantum simulation of the Schr\"odingerization method in \cite{JLMY2025SchOptimal} is shown in Fig.~\ref{fig:schr_circuit}, where $\bb{\psi} = \sum_{k\in [N_p]} \psi(p_k)\ket{k}$. The evolution operator $\mathcal{U}(T)$ can be expressed as a select oracle
\begin{equation*}\label{selectVk}
\mathcal{U}(T) = \sum_{k=0}^{N_p-1} \ket{k}\bra{k} \otimes \e^{-\i \mu_k A T}
=: \sum_{k=0}^{N_p-1} \ket{k}\bra{k} \otimes \mathcal{U}_k(T),
\end{equation*}
where each unitary $\mathcal{U}_k(T)$ corresponds to the simulation of the Hamiltonian $H_{\mu_k} := \mu_k A$, with $\mu_k = \frac{\pi}{R}(k - \frac{N_p}{2})$ being the Fourier modes.

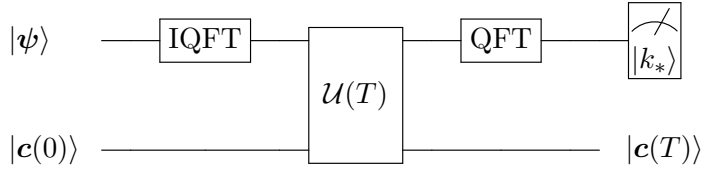
\begin{figure}[H]
\centering
\centerline{
\Qcircuit @C=1em @R=2em {
\lstick{\hbox to 2.7em{$\ket{\bb{\psi}}$\hss}}
& \qw
& \gate{\text{IQFT}}
& \qw
& \multigate{1}{\mathcal{U}(T)}	
& \qw
& \gate{\text{QFT}}
& \qw	
& \qw  &\meterB{\ket{k_*}}\\
\lstick{\hbox to 2.7em{$\ket{\bb{c}(0)}$}}
& \qw
& \qw
& \qw
& \ghost{\mathcal{U}(T)}
& \qw
& \qw
& \qw
& \qw  & \hbox to 2em{$\ket{\bb{c}(T)}$}
}}
\caption{Quantum circuit for the Schr\"odingerization approach in \cite{JLMY2025SchOptimal}.}
\label{fig:schr_circuit}
\end{figure}

Following the similar discussion in \cite{JLMY2025SchOptimal}, we can conclude the following lemma.

\begin{theorem}\label{thm:heat}
Let $\varepsilon>0$ be the precision. Then there exists a quantum algorithm that prepares an $\varepsilon$-approximation of the state $\ket{\bb{u}(T)}$ with $\Omega(1)$ success probability and a flag indicating success, using
	\[\mathcal{O}\Big( \frac{\|\bb{u}(0)\| } {\|\bb{u}(T)\|} d N_x \beta T  \log \frac{\|\bb{u}(0)\|}{\varepsilon \|\bb{u}(T)\|}   \Big)\]
	queries to the block-encoding oracle of $A$.
\end{theorem}
\begin{proof}
According to Theorem 2.2 of \cite{JLMY2025SchOptimal}, there exists a quantum algorithm that prepares an $\varepsilon$-approximation of the state $\ket{\bb{c}(T)}$ with $\Omega(1)$ success probability and a flag indicating success, using
\[\mathcal{O}\Big( \frac{\|\bb{c}(0)\| } {\|\bb{c}(T)\|}\alpha_A T  \log \frac{\|\bb{c}(0)\|}{\varepsilon \|\bb{c}(T)\|}   \Big)\]
queries to the block-encoding oracle of $A$. The proof is concluded by noting that $\alpha_A = d \beta \nu_{\max} = \mathcal{O}(d \beta N_x)$ and $\|\bb{c}(t)\| = \|\bb{u}(t)\|$.
\end{proof}

\subsection{Quantum algorithm for the regularized Wasserstein proximal operator}

\subsubsection{Assembly of the quantum subroutines}

We now assemble the quantum algorithm for the regularized WPO using the heat equation solver (Theorem~\ref{thm:heat}) and the Hadamard product/division protocols (Theorem~\ref{thm:hadamard}). The overall procedure consists of three main steps:
\begin{enumerate}
    \item Solve the heat equation \eqref{heatHJ} to obtain $\eta(T,x)$.
    \item Solve the heat equation \eqref{heatFP} to obtain $\psi(T,x)$.
    \item Combine the results to recover the terminal density $\rho(T,x) = \psi(T,x) \eta_0(x)$.
\end{enumerate}
As noted in Remark~\ref{rem:hadamard}, a difficulty arises in the quantum implementation of Step~2, where the initial condition for $\psi$ requires the componentwise division $\rho_0(x) / \eta(T,x)$. This operation is performed using the quantum Hadamard division protocol developed in Theorem~\ref{thm:hadamard}.

Before proceeding, we specify the required quantum oracles:

\begin{itemize}
    \item \textbf{State preparation oracles.}
    \begin{itemize}
        \item $U_{\rho_0}$: prepares the normalized initial density state $\ket{\bb{\rho}_0} = \sum_j \rho_0(x_j) \ket{j} / \|\bb{\rho}_0\|$.
        \item $U_{\eta_0}$: prepares the state encoding $\e^{-V(x)/(2\beta)}$, i.e., $\ket{\bb{\eta}_0} = \sum_j \e^{-V(x_j)/(2\beta)} \ket{j} / \|\bb{\eta}_0\|$.
    \end{itemize}
    \item \textbf{Block-encoding oracle.}
    \begin{itemize}
        \item $U_A$: an $(\alpha_A, m_A, 0)$-block-encoding of the diagonal matrix $A = \beta \sum_{l=1}^d (\bb{D}_l^\nu)^2$, which represents the discrete Laplacian in the Fourier basis. The construction of $U_A$ follows from the block-encoding of $D_\nu$ via standard linear combination of unitaries (LCU) techniques.
    \end{itemize}
\end{itemize}
We assume these oracles are given and can be queried. The query complexity counts calls to these oracles.

\paragraph{Step 1: Solving for $\eta(T,x)$.}
Using $U_{\eta_0}$, we prepare $\ket{\bb{\eta}_0}$. Applying the heat equation solver (Theorem~\ref{thm:heat}) with the block-encoding $U_A$, we obtain a state $\ket{\bb{\eta}(T)}$ that is $\varepsilon$-close to the normalized exact state $\ket{\bb{\eta}_{\text{exact}}(T)}$, with success probability $\Omega(1)$. The associated unitary that prepares $\ket{\bb{\eta}(T)}$ is denoted as $U_{\eta(T)}$. The cost is
\[
\mathcal{O}\Big( \frac{\|\bb{\eta}_0\|}{\|\bb{\eta}(T)\|} d N_x \beta T \log\frac{\|\bb{\eta}_0\|}{\varepsilon \|\bb{\eta}(T)\|} \Big)
\]
queries to $U_A$ and $\mathcal{O}(\|\bb{\eta}_0\|/\|\bb{\eta}(T)\|)$ queries to $U_{\eta_0}$.

\paragraph{Step 2: Solving for $\psi(T,x)$.}
The initial condition for $\psi$ is $\psi(0,x) = \rho_0(x) / \eta(T,x)$. Using the Hadamard division circuit (Theorem~\ref{thm:hadamard}(ii)), we prepare a state $\ket{\bb{\psi}_0}$ that is $\varepsilon$-close to the normalized exact initial state $\ket{\bb{\psi}_{\text{exact}}(0)}$. This circuit, denoted by $U_{\psi_0}$, requires:
\begin{itemize}
    \item $U_{\rho_0}$ to prepare $\ket{\bb{\rho}_0}$,
    \item $U_{\eta(T)}$ to prepare $\ket{\bb{\eta}(T)}$,
    \item $\mathcal{O}(\kappa_{\bb{\eta}(T)} \log(1/\varepsilon))$ calls to controlled-$U_{\eta(T)}$ and $U_{\rho_0}$, where
    \begin{equation}\label{kappastep2}
    \kappa_{\bb{\eta}(T)} = \frac{\max_j |\eta(T,x_j)| }{ \min_j |\eta(T,x_j)|}
    \end{equation}
         is the condition number of the diagonal operator $D_{\bb{\eta}(T)}$.
\end{itemize}

Once $\ket{\bb{\psi}_0}$ is prepared, we apply the heat equation solver again to obtain $\ket{\bb{\psi}(T)}$ with precision $\varepsilon$. The cost is
\[
\mathcal{O}\left( \frac{\|\bb{\psi}_0\|}{\|\bb{\psi}(T)\|} d  N_x \beta T \log\frac{\|\bb{\psi}_0\|}{\varepsilon \|\bb{\psi}(T)\|} \right)
\]
queries to $U_A$ and $\mathcal{O}(\|\bb{\psi}_0\|/\|\bb{\psi}(T)\|)$ queries to $U_{\psi_0}$.

\paragraph{Step 3: Recovering $\rho(T,x)$.}
The terminal density is $\rho(T,x) = \psi(T,x) \eta_0(x)$. Using the Hadamard product circuit (Theorem~\ref{thm:hadamard}(i)), we compute $\ket{\bb{\rho}(T)}$ from $\ket{\bb{\psi}(T)}$ and the known state $\ket{\bb{\eta}_0}$. This requires $\mathcal{O}(g_{\text{prod}})$ calls to controlled-$U_{\psi(T)}$ and $U_{\eta_0}$, where
\[
g_{\text{prod}} = \frac{\|\bb{\psi}(T)\| \|\bb{\eta}_0\|}{\|\bb{\psi}(T) \odot \bb{\eta}_0\|}.
\]

\subsubsection{The query complexity}

Combining the three steps, the overall algorithm prepares an $\varepsilon$-approximation of the normalized terminal density state $\ket{\bb{\rho}(T)}$ with $\Omega(1)$ success probability. The total query complexity is dominated by the heat equation solvers and the Hadamard division step.

\begin{theorem}[Quantum algorithm for the regularized WPO]
\label{thm:Wasserstein}
Assume the existence of the state preparation oracles $U_{\rho_0}$ and $U_{\eta_0}$, and the block-encoding $U_A$ of the discrete Laplacian. Then there exists a quantum algorithm that prepares an $\varepsilon$-approximation of the normalized state $\ket{\bb{\rho}(T)}$ corresponding to the terminal density $\rho(T,x) = \operatorname{WProx}_{T,\mathcal{F},\beta}(\rho_0)(x)$, with $\Omega(1)$ success probability, using
\begin{itemize}
  \item
  \[
\mathcal{O}\Big( g_{\text{prod}} \frac{\|\bb{\psi}_0\|}{\|\bb{\psi}(T)\|} d  N_x \beta T  \Big( \log\frac{\|\bb{\psi}_0\|}{\varepsilon \|\bb{\psi}(T)\|} + \frac{\|\bb{\eta}_0\|}{\|\bb{\eta}(T)\|} \kappa_{\bb{\eta}(T)} \log\frac{1}{\varepsilon} \log\frac{\|\bb{\eta}_0\|}{\varepsilon \|\bb{\eta}(T)\|} \Big) \Big)
\]
queries to $U_A$,
  \item
  \[
\mathcal{O}\Big( g_{\text{prod}} \frac{\|\bb{\psi}_0\|}{\|\bb{\psi}(T)\|}   \kappa_{\bb{\eta}(T)} \log\frac{1}{\varepsilon} \Big)
\]
queries to $U_{\rho_0}$,
  \item \[
\mathcal{O}\Big( g_{\text{prod}} \frac{\|\bb{\psi}_0\|}{\|\bb{\psi}(T)\|} \frac{\|\bb{\eta}_0\|}{\|\bb{\eta}(T)\|}  \kappa_{\bb{\eta}(T)} \log\frac{1}{\varepsilon} \Big)
\]
queries to $U_{\eta_0}$,
\end{itemize}
where
\[
g_{\text{prod}} = \frac{\|\bb{\psi}(T)\| \|\bb{\eta}_0\|}{\|\bb{\psi}(T) \odot \bb{\eta}_0\|}, \qquad \kappa_{\bb{\eta}(T)} = \frac{\max_j |\eta(T,x_j)|}{\min_j |\eta(T,x_j)|}.
\]
\end{theorem}



The query complexity depends on certain normalization factors and on the condition number $\kappa_{\bb{\eta}(T)}$. It is therefore important to assess whether these ``heat equation quantities'' remain moderate; otherwise, the resulting bounds may fail to be efficient with respect to the parameters $d$ and $N_x$.

For simplicity, we take the spatial domain to be the periodic box $\mathbb{T}^d = [0,L]^d$ in what follows.

\begin{lemma}
[Estimates for the heat equation quantities]
\label{lem:estimates}
Let $\eta(t,x)$ and $\psi(t,x)$ be the solutions to the heat equations \eqref{heatHJ} and \eqref{heatFP}, respectively, on the periodic domain $\mathbb{T}^d = [0,L]^d$ with periodic boundary conditions.
Assume that the potential $V$ satisfies $0 \le V(x) \le V^*$ for all $x \in \mathbb{T}^d$,
and that $\rho_0$ is a normalized probability density, i.e., $\int_{\mathbb{T}^d} \rho_0(x)\d x = 1$.
Then the following estimates hold:
\[
\frac{\|\bb{\eta}_0\|}{\|\bb{\eta}(T)\|} \lesssim \e^{V^*/(2\beta)}, \quad
\frac{\|\bb{\psi}_0\|}{\|\bb{\psi}(T)\|}\lesssim \e^{V^*/(2\beta)} L^{d/2}  \|\rho_0\|_{L^2},\]
\[
g_{\mathrm{prod}} \lesssim (L \Delta x)^{d/2} \e^{V^*/\beta}, \quad
\kappa_{\bb{\eta}(T)} \le \e^{V^*/(2\beta)}.
\]
\end{lemma}
\begin{proof}
\textbf{Estimate of $\frac{\|\bb{\eta}_0\|}{\|\bb{\eta}(T)\|}$.}
Noting that
\[\frac{\|\bb{\eta}_0\|}{\|\bb{\eta}(T)\|} \approx \frac{\|\eta(0)\|_{L^2}}{\|\eta(T)\|_{L^2}},\]
we only need to derive an upper bound for the right-hand side.
Consider the heat equation \eqref{heatHJ} for $\eta$ on the periodic domain $\mathbb{T}^d = [0,L]^d$ with periodic boundary conditions.

First, we note that the spatial average of $\eta$ is conserved. Indeed,
\[
\frac{\d}{\d t}\int_{\mathbb{T}^d} \eta(t,x)\d x
= \beta \int_{\mathbb{T}^d} \Delta \eta(t,x)\d x = 0,
\]
where the last equality follows from the periodic boundary conditions. Hence,
\[
    \bar{\eta}(t) := \frac{1}{L^d}\int_{\mathbb{T}^d} \eta(t,x)\d x = \bar{\eta}_0,
\]
where
\[
    \bar{\eta}_0 := \frac{1}{L^d}\int_{\mathbb{T}^d} \eta_0(x)\d x.
\]

We decompose $\eta$ into its mean and zero-mean parts:
\[
    \eta(t,x) = \bar{\eta}_0 + u(t,x), \qquad
    \int_{\mathbb{T}^d} u(t,x)\d x = 0.
\]
Since $\bar{\eta}_0$ is constant in space and time, $u$ satisfies the same heat equation:
\[
    \partial_t u = \beta \Delta u, \qquad u(0,x) = \eta_0(x) - \bar{\eta}_0.
\]
The $L^2$-norm of $\eta$ can be decomposed as
\begin{equation} \label{decomL2}
    \|\eta(t)\|_{L^2}^2 = L^d |\bar{\eta}_0|^2 + \|u(t)\|_{L^2}^2,
\end{equation}
because the mean and zero-mean parts are orthogonal.

To estimate $\|u(t)\|_{L^2}$, we multiply the heat equation for $u$ by $u$ and
integrate over $\mathbb{T}^d$:
\begin{equation}\label{dudtL2}
    \frac{1}{2}\frac{\d}{\d t}\|u(t)\|_{L^2}^2
    = \beta \int_{\mathbb{T}^d} u \Delta u \d x
    = -\beta \|\nabla u(t)\|_{L^2}^2,
\end{equation}
where the integration by parts has no boundary terms due to the periodic conditions.
Thus, $\|u(t)\|_{L^2}$ is non-increasing:
\[
    \|u(t)\|_{L^2} \le \|u(0)\|_{L^2}.
\]
Consequently, from \eqref{decomL2} we obtain the monotonicity of the full norm:
\begin{equation}
    \|\eta(t)\|_{L^2} \le \|\eta(0)\|_{L^2}.
\end{equation}

Moreover, on the periodic domain $\mathbb{T}^d$, the Poincar\'e inequality for zero-mean functions gives
\[
    \|\nabla u\|_{L^2}^2 \ge \lambda_1 \|u\|_{L^2}^2,
    \qquad \lambda_1 = \Big(\frac{2\pi}{L}\Big)^2,
\]
where $\lambda_1$ is the smallest nonzero eigenvalue of $-\Delta$ on the periodic domain. Substituting this into \eqref{dudtL2} yields
\begin{equation}
    \frac{\d}{\d t}\|u(t)\|_{L^2}^2 \le -2\beta \lambda_1 \|u(t)\|_{L^2}^2,
\end{equation}
which implies the exponential decay
\begin{equation}\label{uexp}
    \|u(t)\|_{L^2} \le \e^{-\beta \lambda_1 t} \|u(0)\|_{L^2}.
\end{equation}
Combining \eqref{decomL2} and \eqref{uexp}, we obtain a lower bound for $\|\eta(t)\|_{L^2}$:
\[
    \|\eta(t)\|_{L^2}^2
    \ge L^d |\bar{\eta}_0|^2
    + \e^{-2\beta \lambda_1 t} \|u(0)\|_{L^2}^2.
\]
Since $\|u(0)\|_{L^2}^2 = \|\eta_0\|_{L^2}^2 - L^d |\bar{\eta}_0|^2$,
we have
\begin{equation}
    \|\eta(t)\|_{L^2}^2
    \ge L^d |\bar{\eta}_0|^2
    + \e^{-2\beta \lambda_1 t}
    \Big(\|\eta_0\|_{L^2}^2 - L^d |\bar{\eta}_0|^2\Big).
\end{equation}

Taking square roots and rearranging, we arrive at the desired upper bound:
\[
\frac{\|\eta(0)\|_{L^2}}{\|\eta(t)\|_{L^2}}
        \le
        \frac{\|\eta_0\|_{L^2}}{
            \sqrt{
                L^d |\bar{\eta}_0|^2
                + \e^{-2\beta \lambda_1 t}
                (\|\eta_0\|_{L^2}^2 - L^d |\bar{\eta}_0|^2)
            }
        }.
\]
In particular, since the second term in the denominator is nonnegative,
we have the simpler uniform bound
\begin{equation}\label{ratetabound}
        \frac{\|\eta(0)\|_{L^2}}{\|\eta(T)\|_{L^2}}
        \le \frac{\|\eta_0\|_{L^2}}{\sqrt{L^d}\,|\bar{\eta}_0|}.
\end{equation}

Under the assumption that the potential $V$ satisfies
$0 \le V(x) \le V^*$ for all $x \in \mathbb{T}^d$, the uniform bound in \eqref{ratetabound}
can be made more explicit. Indeed, since $\e^{-V^*/(2\beta)} \le \e^{-V(x)/(2\beta)} \le 1$,
we have
\[
\bar{\eta}_0 = \frac{1}{L^d}\int_{\mathbb{T}^d} \e^{-V(x)/(2\beta)}\d x
\ge \e^{-V^*/(2\beta)}.
\]
Also, $\eta_0(x) \le 1$ implies $\|\eta_0\|_{L^2} \le \sqrt{L^d}$. Therefore,
\[
\frac{\|\eta_0\|_{L^2}}{\|\eta(T)\|_{L^2}} \le \frac{\|\eta_0\|_{L^2}}{\sqrt{L^d}\,|\bar{\eta}_0|}
\le \e^{V^*/(2\beta)}.
\]

\textbf{Estimate of $\frac{\|\bb{\psi}_0\|}{\|\bb{\psi}(T)\|}$.}
The ratio $\|\bb{\psi}_0\| / \|\bb{\psi}(T)\|$ can be estimated
using the same argument as for $\eta$, which gives
\[
\frac{\|\bb{\psi}_0\|}{\|\bb{\psi}(T)\|} \approx
\frac{\|\psi_0\|_{L^2}}{\|\psi(T)\|_{L^2}}
\le \frac{\|\psi_0\|_{L^2}}{\sqrt{L^d}\,|\bar{\psi}_0|},
\]
where $\bar{\psi}_0 = L^{-d}\int \psi(0,x)\d x$. Using $\psi(0,x) = \rho_0(x)/\eta(T,x)$ and the fact that
$\eta(T,x) \le 1$ (by the maximum principle, since $\eta_0 \le 1$),
we have $\bar{\psi}_0 \ge \bar{\rho}_0 = L^{-d}\int \rho_0\d x$.
Also, $\|\psi_0\|_{L^2} \le \e^{V^*/(2\beta)}\|\rho_0\|_{L^2}$. Therefore,
\[
\frac{\|\psi_0\|_{L^2}}{\|\psi(T)\|_{L^2}}
\le \e^{V^*/(2\beta)} \frac{\|\rho_0\|}{\sqrt{L^d}\,\bar{\rho}_0}.
\]
Since $\rho_0$ is a normalized probability density ($\int \rho_0 \d x = 1$), this reduces to
\[
\frac{\|\psi_0\|_{L^2}}{\|\psi(T)\|_{L^2}}
\le \e^{V^*/(2\beta)} \sqrt{L^d} \|\rho_0\|_{L^2}.
\]

\textbf{Estimate of $g_{\text{prod}}$.}
The factor $g_{\text{prod}}$ appearing in the Hadamard product complexity
can be estimated as follows. In the continuous setting,
\[
g_{\text{prod}}
\approx \Delta x^{d/2} \frac{\|\psi(T)\|_{L^2} \|\eta_0\|_{L^2}}{\|\psi(T) \cdot \eta_0\|_{L^2}}.
\]
Since $\psi(T,x) = \rho(T,x)\eta_0(x)$ and
$0 < m \le \eta_0(x) \le 1$ with $m = \e^{-V^*/(2\beta)}$, we have
\[
\|\psi(T)\|_{L^2} = \|\rho(T)\eta_0\|_{L^2} \le \|\rho(T)\|_{L^2},
\]
and
\[
\|\psi(T) \cdot \eta_0\|_{L^2} = \|\rho(T)\eta_0^2\|_{L^2} \ge m^2 \|\rho(T)\|_{L^2}.
\]
Therefore,
\[
g_{\text{prod}}
\lesssim \Delta x^{d/2} \frac{\|\eta_0\|_{L^2}}{m^2}.
\]
Since $\|\eta_0\|_{L^2}^2 = \int \e^{-V(x)/\beta} \d x \le L^d$, we obtain
\[
g_{\text{prod}} \lesssim \Delta x^{d/2} \sqrt{L^d}\, \e^{V^*/\beta}
= (L \Delta x)^{d/2} \e^{V^*/\beta}.
\]

\textbf{Estimate of $\kappa_{\bb{\eta}(T)}$.}
This quantity governs the query complexity of the quantum Hadamard division protocol.

For the initial data $\eta_0(x) = \e^{-V(x)/(2\beta)}$ with $0 \le V(x) \le V^*$, the maximum principle gives
\[
\min_x \eta(T,x) \ge \min_x \eta_0(x) = \e^{-V^*/(2\beta)},
\qquad
\max_x \eta(T,x) \le \max_x \eta_0(x) = 1.
\]
Therefore,
\[
\kappa_{\bb{\eta}(T)} \le \frac{\max_x \eta_0(x)}{\min_x \eta_0(x)}
= \e^{V^*/(2\beta)}.
\]
This completes the proof.
\end{proof}

\begin{theorem}[ Query complexity under minimizer-invariance scaling]
\label{thm:Wassersteinqc}
Suppose that the potential $V$ satisfies $0 \le V(x) \le V^*$ for all $x \in \mathbb{T}^d$.
By choosing the grid spacing such that $L^2 \Delta x \le 1$, and noting that the WPO is invariant under the scaling
\[
V \mapsto \frac{V}{G}, \qquad T \mapsto TG,
\]
for any $G>0$, we may take $G = 1/\beta$ without changing the minimizer $\rho_T$. Under this scaling, there exists a quantum algorithm that prepares an $\varepsilon$-approximation of the normalized state $\ket{\bb{\rho}(T)}$ corresponding to the terminal density $\rho(T,x) = \operatorname{WProx}_{T,\mathcal{F},\beta}(\rho_0)(x)$, with $\Omega(1)$ success probability, using
  \[
\mathcal{O}\Big( C_* d N_x T \log^2\frac{1}{\varepsilon} \Big)
\]
queries to $U_A$, and
  \[
\mathcal{O}\Big( C_* \log\frac{1}{\varepsilon} \Big)
\]
queries to $U_{\rho_0}$ and $U_{\eta_0}$,
where  $C_* = \e^{3V^*/2} \|\rho_0\|_{L^2}$ is a constant.
\end{theorem}

\begin{proof}
The proof proceeds by analyzing the scaling behavior of the factors appearing in the complexity estimate of Theorem~\ref{thm:Wasserstein}. We show that, under a suitable rescaling of the problem parameters, all factors remain bounded by constants independent of the discretization and the regularization parameter $\beta$.

\paragraph{Cancellation of the domain-size dependence.}
In the complexity estimate for the Hadamard product step, the factors
$g_{\mathrm{prod}}$ and $\|\bb{\psi}_0\| / \|\bb{\psi}(T)\|$
appear multiplicatively. Combining the estimates from Lemma~\ref{lem:estimates},
\[
g_{\mathrm{prod}} \lesssim (L\Delta x)^{d/2} \e^{V^*/\beta}, \qquad
\frac{\|\bb{\psi}_0\|}{\|\bb{\psi}(T)\|}
\lesssim \e^{V^*/(2\beta)} L^{d/2} \|\rho_0\|_{L^2},
\]
we obtain
\[
g_{\mathrm{prod}} \cdot \frac{\|\bb{\psi}_0\|}{\|\bb{\psi}(T)\|}
\lesssim (L^2 \Delta x)^{d/2} \e^{3V^*/(2\beta)} \|\rho_0\|_{L^2}.
\]
By choosing the grid spacing such that $L^2 \Delta x \le 1$, the factor
$(L^2 \Delta x)^{d/2}$ is bounded by a constant. Hence the combined factor
eliminates the exponential growth with respect to the domain size, and only data-dependent constants remain.

\paragraph{Removal of $\beta$-dependence via minimizer-invariance scaling.}
The estimates in Lemma~\ref{lem:estimates} depend on $\beta$ through the factor $\e^{V^*/(2\beta)}$. However, the WPO is invariant under the simultaneous scaling
\[
V \mapsto \frac{V}{G}, \qquad T \mapsto TG
\]
for any $G > 0$, since this corresponds to dividing the objective functional \eqref{Wprox} by $G$ and thus does not change the minimizer $\rho_T$. Under this transformation, the estimate for $g_{\mathrm{prod}}$ in Lemma~\ref{lem:estimates} becomes
\[
g_{\mathrm{prod}} \lesssim (L\Delta x)^{d/2} \e^{V^*/(\beta G)}.
\]
Choosing $G = 1/\beta$ yields
\[
g_{\mathrm{prod}} \lesssim (L\Delta x)^{d/2} \e^{V^*}.
\]
Similarly, the other factors become
\[
\frac{\|\bb{\eta}_0\|}{\|\bb{\eta}(T)\|}
\lesssim \e^{V^*/2},\qquad
\kappa_{\bb{\eta}(T)} \le \e^{V^*/2}.
\]
Thus, all exponential factors are now independent of $\beta$. The evolution time increases from $T$ to $T/\beta$, but this is exactly compensated since the heat equation solver complexity in Theorem~\ref{thm:Wasserstein} contains a factor $\beta T$, resulting in no net additional cost.

Combining the above estimates with Theorem~\ref{thm:Wasserstein} yields the desired complexity bound.
\end{proof}

\paragraph{Classical computational cost.}
 For comparison, we briefly discuss the classical complexity of solving the regularized WPO using the discretization described in Section~\ref{subsec:Schr}. The main computational task consists of solving two heat equations of the form \eqref{heatbeta} on a $d$-dimensional grid with $N = N_x^d$ points.

Using a Fourier spectral method, each heat equation can be solved in $\mathcal{O}(N \log N)$ operations per time step via fast Fourier transforms. In addition, the Hadamard division $\bb{\rho}_0 \oslash \bb{\eta}(T)$ and Hadamard product $\bb{\psi}(T) \odot \bb{\eta}_0$ each require $\mathcal{O}(N)$ operations. Therefore, the classical cost {per time step} scales as
\[
\mathcal{O}\bigl(N \log N\bigr) = \mathcal{O}\bigl(d N_x^d \log N_x\bigr),
\]
up to problem-dependent constants.

This classical cost exhibits an exponential dependence on the spatial dimension $d$ through the factor $N_x^d$, which becomes prohibitive for high-dimensional problems. In contrast, the quantum algorithm in Theorem~\ref{thm:Wassersteinqc} achieves a query complexity that is polynomial in $d$ and $N_x$ (specifically, linear in $d N_x$), yielding an exponential speedup over classical methods. This exponential improvement is the primary advantage of our quantum approach for high-dimensional regularized WPOs.

\begin{remark}
The classical cost stated above follows the Fourier spectral discretization in Section~\ref{subsec:Schr}. An alternative classical approach, based on the kernel integral formula (Proposition~5 of \cite{LLO2023}), would require evaluating the kernel
$K(x,y)$ for all pairs of grid points, resulting in a cost of $\mathcal{O}(N^2) = \mathcal{O}(N_x^{2d})$ per iteration, which is even more expensive for large $d$.
\end{remark}

\section{Numerical experiment} \label{sec:num}

We consider Example A from \cite{LLO2023}. The parameter setup is as follows:
\[
\Omega = [-b, b]^d, \qquad V(x) = \e^{-|x+0.25|^2/0.5}, \qquad \rho_0(x) = \frac{1}{\sigma_0\sqrt{2\pi}} \e^{-|x-0.25|^2/(2\sigma_0^2)}, \quad \sigma_0 = 0.1.
\]
For the unitary $\mathcal{U}_k(T)= \e^{-\i \mu_k A T}$, one can employ the quantum circuits detailed in~\cite{Sato24Circuit, HuJin24SchrCircuit, JLY24Circuits}, which are implemented using UnitaryLab, an open-source quantum computing framework at \url{http://unitarylab.com}.

For the 1D case, we set $b = 5$ and use $N_x = 2^8$ spatial grid points. The numerical solutions are displayed in Fig.~\ref{fig:Wasserstein}. We examine the effects of the coefficients $\beta$ and $T$ on the resulting density distribution. In the left panel, we vary $T \in \{0.1, 0.2, 0.5\}$ with $\beta = 0.25$ fixed; in the right panel, we vary $\beta \in \{0.125, 0.25, 0.5\}$ with $T = 0.2$ fixed. The results are consistent with those reported in Fig.~1 of \cite{LLO2023}.

\begin{figure}[!htb]
  \centering
  \includegraphics[width=0.4\textwidth]{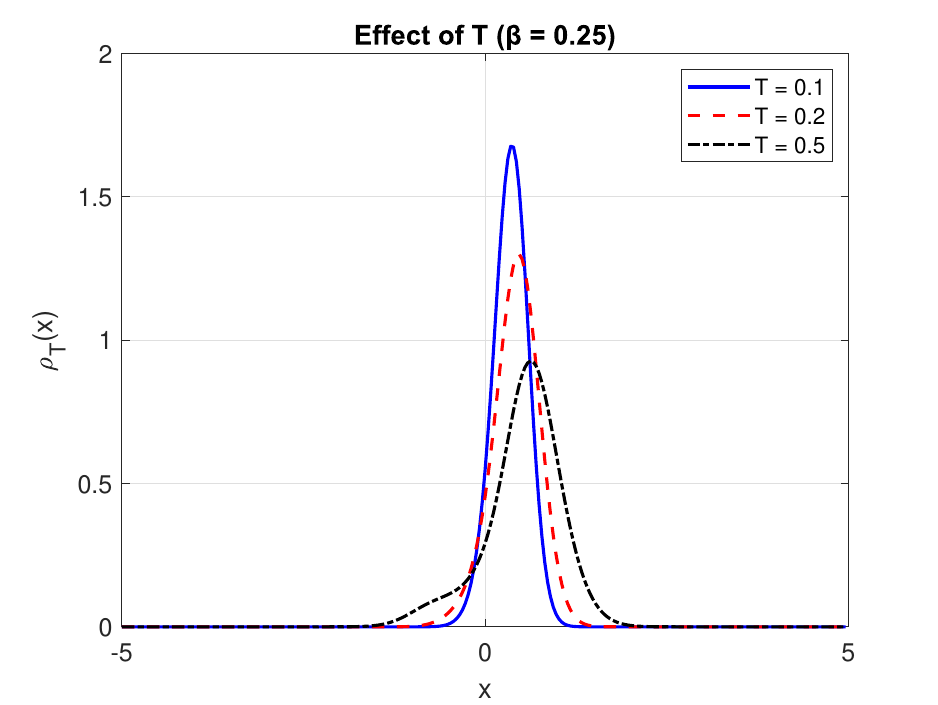}
  \includegraphics[width=0.4\textwidth]{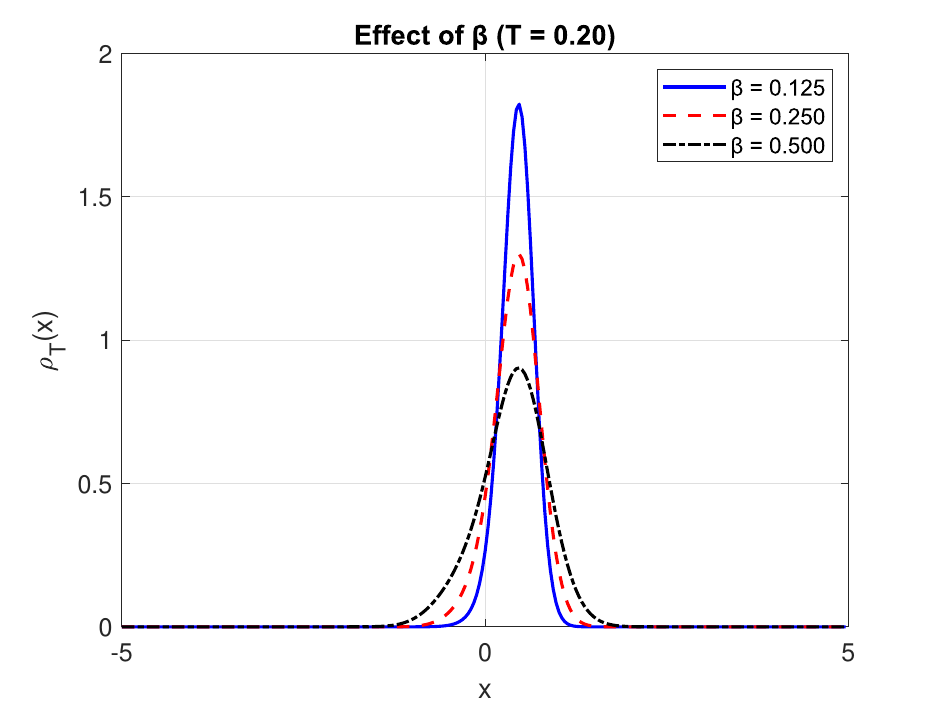}\\
  \caption{Numerical results for the regularized Wasserstein proximal problem in 1D.}\label{fig:Wasserstein}
\end{figure}

We repeat the test for the 2D case with $b = 1.5$, $N_x = 2^5$, $\beta = 0.25$, and $T = 0.2$. The results are shown in Fig.~\ref{fig:Wasserstein2D}. For comparison, we also provide the result obtained via the kernel integral formula (cf. Proposition~5 of \cite{LLO2023}). The two sets of results agree well.

\begin{figure}[!htb]
  \centering
  \includegraphics[width=0.4\textwidth]{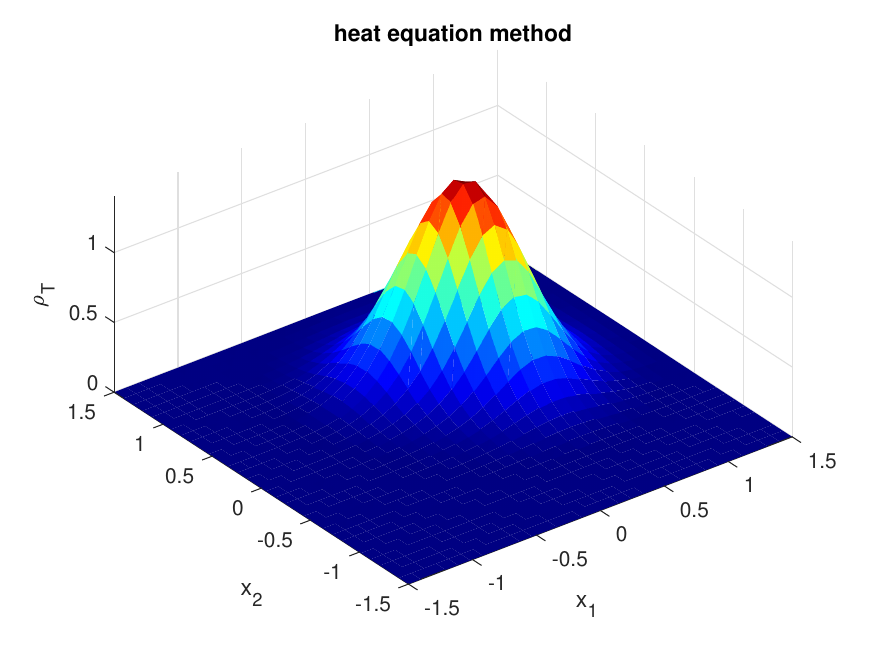}
  \includegraphics[width=0.4\textwidth]{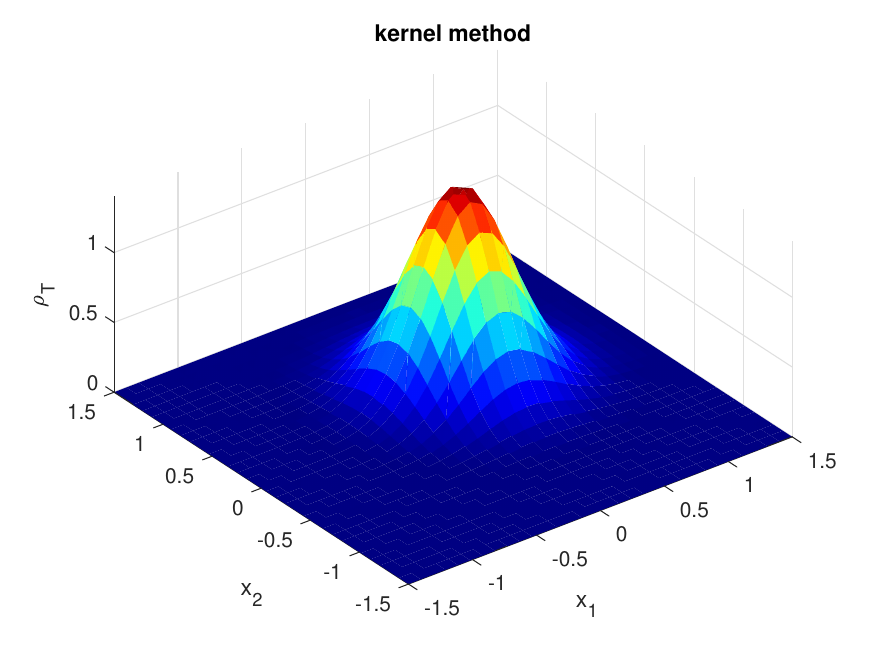}\\
  \caption{Numerical results for the regularized Wasserstein proximal problem in 2D.}\label{fig:Wasserstein2D}
\end{figure}

\section{Conclusion and outlook}\label{sec:conclusion}

In this work, we have developed efficient quantum algorithms for the regularized Wasserstein proximal operators. Building on the work of \cite{LLO2023},  the forward-backward PDE system is first converted to forward heat equations via the Cole-Hopf transform, which is then combined with Schr\"odingerization so becomes suitable for quantum simulation. We constructed quantum circuits for Hadamard division and product operations based on a simple matrix-vector multiplication representation, enabling the preparation of initial data for the second heat equation and the recovery of the terminal density state. The complete quantum algorithm prepares an approximation of the terminal density state with exponential speedups.

Due to the broad applications of the Wasserstain proximal operators, the quantum algorithm developed in this work opens up several promising directions for future research. Here we list a few possibilities:

\begin{itemize}
\item In \textbf{gradient flows and nonlinear PDEs}, it could accelerate the JKO scheme for evoluting  PDEs by time-discretizing the Wasserstein gradient flow, with applications in, for examples, porous media flows and thin film equations \cite{jordan1998variational, LiLuWang2020}.

\item In \textbf{generative AI}, the connection between WPOs and diffusion models \cite{lin2021wasserstein, Zhang2024WPO} suggests that our quantum algorithms may offer a pathway toward efficient quantum generative modeling.
\item In \textbf{mean-field games} \cite{LasryLions2007,HuangMalhameCaines2006, Gueant2012Quadratic, GueantLasryLions2010}, the WPO serves as a fundamental building block for computing equilibrium dynamics of large interacting populations; our approach could potentially provide quantum acceleration  in high-dimensional settings.
\item In \textbf{stochastic optimal control}, the close connection between the regularized WPO and the Schr\"odinger bridge problem \cite{Leonard2012, ChenGeorgiouPavon2016, Caluya2022SBP} points to possible extensions of our framework to density control problems.
\end{itemize}

The extension of the present algorithm to \textbf{nonlinear energy functionals}~--~namely, replacing the linear functional $\mathcal{F}(q)$ in \eqref{Wprox} with a nonlinear counterpart~--~also remains for future work.


\section*{Acknowledgments}

SJ and NL acknowledge the support of the NSFC grant No. 12341104, the Shanghai Pilot Program for Basic Research,  the Science and Technology Commission of Shanghai Municipality (STCSM) grant no. 24LZ1401200 (21JC1402900), the Major Basic Research Project of the Shanghai Jiao Tong University 2030 Initiative, and the Fundamental Research Funds for the Central Universities. NL is also supported by grant NSFC No. 12471411.
YY was supported by NSFC grant (No.\ 12301561), the Key Project of Scientific Research Project of Hunan Provincial Department of Education (No.\ 24A0100), the Science and Technology Innovation Program of Hunan Province (No.\ 2025RC3150) and the general program of Hunan Provincial Natural Science Foundation (No.\ 2026JJ50003).
This research was supported in part by the 111 Project (No.\ D23017), and Program for Science and Technology Innovative Research Team in Higher Educational Institutions of Hunan Province of China.


\newcommand{\etalchar}[1]{$^{#1}$}

\end{document}